\documentclass[letterpaper, 12pt]{amsart}

\usepackage{tikz}

\usetikzlibrary{shadings}

\usepackage{amscd}

\usepackage{amsmath}

\usepackage{amsfonts}

\usepackage{amssymb}

\usepackage{amsthm}

\usepackage[utf8]{inputenc}
\usepackage[T1]{fontenc}
\usepackage[charter]{mathdesign} 
\usepackage{bbold}

\usepackage{mathtools}

\usepackage{arydshln}

\usepackage{fancyhdr}

\usepackage{verbatim}

\usepackage[all]{xy}

\usepackage{enumitem}

\usepackage{color}

\usepackage[colorlinks=true, linkcolor=teal, citecolor=-teal,
pagebackref=true]{hyperref}

\usepackage[reftex]{theoremref}

\usepackage{nameref}

\usepackage[roundcorner=5pt,  skipabove=10pt, skipbelow=10pt]{mdframed}

\usepackage[colorinlistoftodos]{todonotes}

\theoremstyle{plain}
\newtheorem*{theorem*}{Theorem}
\newtheorem{theorem}{Theorem}[section]
\newtheorem{lemma}[theorem]{Lemma}
\newtheorem{corollary}[theorem]{Corollary}
\newtheorem{proposition}[theorem]{Proposition}

\newtheorem*{claim*}{Claim}

\theoremstyle{definition}
\newtheorem{problem}[theorem]{Problem}
\newtheorem*{example*}{Example}

\theoremstyle{remark}
\newtheorem*{remark*}{Remark}

\newcommand{\ba}{\mathbf a}

\newcommand{\bh}{\mathbf h}

\newcommand{\bw}{\mathbf w}

\newcommand{\bv}{\mathbf v}

\newcommand{\bn}{\mathbf n}

\newcommand{\bp}{\mathbf p}

\newcommand{\bx}{\mathbf x}

\newcommand{\by}{\mathbf y}

\newcommand{\CC}{\mathbb{C}}

\newcommand{\PP}{\mathbb{P}}

\newcommand{\RR}{\mathbb{R}}

\newcommand{\QQ}{\mathbb{Q}}

\newcommand{\NN}{\mathbb{N}}

\newcommand{\TT}{\mathbb{T}}

\newcommand{\ZZ}{\mathbb{Z}}

\newcommand{\calI}{\mathcal{I}}

\newcommand{\calJ}{\mathcal{J}}

\newcommand{\calP}{\mathcal{P}}

\newcommand{\bone}{\mathbb{1}}

\newcommand{\bzero}{\mathbf{0}}

\renewcommand{\leq}{\leqslant}

\renewcommand{\geq}{\geqslant}

\renewcommand{\subset}{\subseteq}

\DeclarePairedDelimiter{\abs}{\lvert}{\rvert}
\DeclarePairedDelimiter{\inner}{\langle}{\rangle}
\DeclarePairedDelimiter{\norm}{\lVert}{\rVert}
\DeclarePairedDelimiter{\floor}{\lfloor}{\rfloor}
\DeclarePairedDelimiter{\set}{\lbrace}{\rbrace}
\DeclarePairedDelimiter{\parens}{\lparen}{\rparen}
\DeclarePairedDelimiter{\brackets}{\lbrack}{\rbrack}

\DeclareMathOperator{\lcm}{lcm}
\DeclareMathOperator{\lex}{lex}
\DeclareMathOperator{\supp}{supp}
\DeclareMathOperator{\Span}{span}

\def\eps{{\varepsilon}}

\def\1int{{[0,1]}}

\makeatletter
\let\@wraptoccontribs\wraptoccontribs
\makeatother

\title[Duffin--Schaeffer examples, real residue systems, and Bohr-set
primes]{Duffin--Schaeffer examples, real residue systems, and Bohr-set
  primes}

\author{Stefan M.~Hesseling \and Felipe A.~Ram{\'i}rez}

\contrib[with an appendix by]{Manuel Hauke}

\address{Department of Mathematics and Computer Science, Wesleyan
  University, Connecticut, USA}

\email{mhesseling@wesleyan.edu}

\email{framirez@wesleyan.edu}

\date{}



\begin{document}

\begin{abstract}
  We prove the following generalization of a well-known result of
  Duffin and Schaeffer: For any given countable sets $Y \subset\RR$
  and $Z\subset\RR\setminus\Span_\QQ(\set{1}\cup Y)$, there exist
  functions $\psi$ such that the set of inhomogeneously
  $\psi$-approximable numbers has zero measure or full measure,
  according as the inhomogeneous parameter lies in $Y$ or $Z$. The
  proof uses an analogue of residue systems where the residues can
  take arbitrary real values, and it also requires information about
  the distribution of primes lying in Bohr sets. We extend a theorem
  of Rogers to the more general real residues setting, and we extend
  Siegel--Walfisz's prime number theorem for arithmetic progressions
  to the setting of prime numbers lying in Bohr sets. We also prove
  that circle rotations equidistribute when sampled along such primes,
  provided the rotation angle is rationally independent of the Bohr
  set parameter, generalizing a theorem of Vinogradov.

  An appendix by Manuel Hauke answers a combinatorial question that is
  posed in the introduction.
\end{abstract}

\maketitle
  

\section{Introduction}

\subsection{Inhomogeneous Duffin--Schaeffer-type examples}

For a real number $y\in\RR$ and a function $\psi:\NN\to\RR_{\geq 0}$,
we study
\begin{equation*}
  W(\psi,y) = \set*{x\in[0,1] : \norm{qx - y} < \psi(q)\textrm{ infinitely often}},
\end{equation*}
the set of inhomogeneously $\psi$-approximable numbers with
inhomogeneous parameter $y$, where $\norm{\cdot}$ denotes distance to
$\ZZ$.

It is a classical theorem of Sz{\"u}sz~\cite[1958]{SzuszinhomKT} that
the Lebesgue measure $\mu(W(\psi, y))$ is $0$ if the series
$\sum\psi(q)$ converges, and $1$ if it diverges and $\psi$ is
nonincreasing. This result is usually referred to as the inhomogeneous
Khintchine theorem, since it generalizes Khintchine's
theorem~\cite[1924]{Khintchine} to the $y\neq 0$ case. Duffin and
Schaeffer~\cite[1941]{duffinschaeffer} showed by example that the
monotonicity condition cannot be omitted from Khintchine's theorem,
and the second author~\cite[2017]{ds_counterex} showed that it cannot
be omitted from Sz{\"u}sz's inhomogeneous result.

More specifically,~\cite[Theorem~1]{ds_counterex} showed that for any
countable set $Y\subset\RR$, one can find a function
$\psi:\NN\to\RR_{\geq 0}$ such that $\sum\psi(q)$ diverges, while
$\mu(W(\psi, y))=0$ for every $y\in Y$. In
addition,~\cite[Theorem~2]{ds_counterex} showed---conditionally---that
if $Z\subset\RR\setminus\Span_\QQ(\set{1}\cup Y)$ is another countable set,
then one can find $\psi$ as before, and with the further property that
$\mu(W(\psi, z))=1$ for every $z\in Z$. The condition was a
conjectural dynamical statement~\cite[Conjecture~17]{ds_counterex}
which could be seen as analogous to Erd{\H o}s's famous covering
systems problems~\cite{Erdos2,Erdoscongruences} (to be discussed).

In this paper, we prove the following \emph{unconditional} version
of~\cite[Theorems~1 and~2]{ds_counterex}.

\begin{theorem}[Duffin--Schaeffer examples with prescribed
  behavior]\th\label{prescription}
  Let $Y,Z \subset\RR$ be countable sets such that
  $Z\subset\RR\setminus\Span_\QQ(\set{1}\cup Y)$. Then there exists
  $\psi:\NN\to \RR_{\geq 0}$ such that
  \begin{equation*}
    \mu(W(\psi, y)) = 0 \qquad \textrm{and} \qquad \mu(W(\psi, z)) = 1
  \end{equation*}
  for every $y\in Y$ and $z\in Z$.
\end{theorem}

The proof is a construction. The functions $\psi$ will be supported in
blocks of integers that are products of elements of Bohr sets, that
is, (judiciously chosen) sets of the form
\begin{equation*}
  N_{\bv}(\by, \eps)  = \set{n \in\NN: \norm{n\by - \bv} < \eps},
\end{equation*}
where $\by,\bv\in \RR^\ell$ and $\eps>0$ and $\norm{\cdot}$ denotes
distance to $\ZZ^\ell$. Loosely speaking, the Bohr set condition leads
to small measure for $W(\psi,y)$, and the remaining challenge is to
show large measure for $W(\psi,z)$. This is achieved in part by
addressing the aforementioned obstacle related to covering systems,
leading to~\th~\ref{thm:rogersext} below, a generalization of a
theorem of Rogers. In the use of~\th~\ref{thm:rogersext} to get lower
measure-bounds, it is convenient if the support of $\psi$ consists of
square-free integers. This leads us to prove
\th~\ref{DirichletPNT,thm:pbohrequid}. The first generalizes
Siegel--Walfisz's and Dirichlet's theorems for prime numbers lying in
arithmetic progressions to primes lying in Bohr sets; the second
generalizes a Vinogradov's theorem~\cite{Vinogradov,Vaughan,Rhin} on
the equidistribution modulo one of $(p z)_{p \in\PP}$, where $\PP$
denotes the primes, to the equidistribution of $(p z)_{p\in\PP\cap N}$
where $N$ is some Bohr set.

\subsection{Real residue systems}

To a pair of $\ell$-tuples $\bn = (n_1, \dots, n_\ell)\in\NN^\ell$ and
$\ba = (a_1, \dots, a_\ell)\in\ZZ^\ell$ we associate the collection of
arithmetic progressions $\set{a_i + n_i\ZZ}_{i=1, \dots, \ell}$ and
refer to it as a \emph{residue system}. If
\begin{equation}\label{eq:coveringsystem}
  \ZZ = \bigcup_{i=1}^\ell \parens*{a_i + n_i\ZZ},
\end{equation}
then it is called a \emph{covering system}. Erd{\H o}s introduced
covering systems in~\cite[1950]{Erdos2}, and several questions about
them have attracted much attention over the years. The most well-known
of these questions was the minimum modulus problem---\emph{Do there
  exist covering systems with $n_1 < n_2 < \dots < n_\ell$ and $n_1$
  arbitrarily large?}---, but there have been many others. After work
of Filaseta \emph{et al.}  where the asymptotic densities of residue
systems were studied~\cite{FFKPY}, Hough answered the minimum modulus
question in the negative~\cite{Hough}. We encourage the reader to
see~\cite{FFKPY,Hough,Balisteretal} and the references therein for a
detailed account of the ongoing history of covering systems. One
result that is especially relevant here is due to Rogers, and it
states that
\begin{equation*}
  d\parens*{\bigcup_{i=1}^\ell \parens*{a_i + n_i\ZZ}} \geq   d\parens*{\bigcup_{i=1}^\ell n_i\ZZ}
\end{equation*}
for all $\ba, \bn$, where $d$ denotes asymptotic
density~\cite[Page~242]{HalberstamRoth}.  In other words, the
asymptotic density of a residue system is minimized when $\ba=\bzero$.

We will make use of an analogue of residue systems where the residues
are allowed to take arbitrary real values. To a pair of $\ell$-tuples
\begin{equation*}
  \bn = (n_1, \dots, n_\ell)\in\NN^\ell \qquad\textrm{and}\qquad \ba = (a_1, \dots, a_\ell)\in\RR^\ell,
\end{equation*}
and a real number $\eps>0$,
we associate the \emph{real residue system}
\begin{equation*}
  R(\eps, \ba, \bn) =\bigcup_{i=1}^\ell \parens*{\brackets*{a_i-\eps, a_i+\eps} + n_i \ZZ}.
\end{equation*}
Notice that the pair $(\ba, \bn)\in\ZZ^{\ell}\times\NN^\ell$ defines a
covering system if and only if $R(1/2,\ba,\bn)=\RR$. Notice, also,
that if $(\ba, \bn)\in \ZZ^{\ell}\times\NN^{\ell}$, then
\begin{equation*}
  d\parens*{\bigcup_{i=1}^\ell \parens*{a_i + n_i\ZZ}}  = \frac{1}{\lcm(\bn)}\mu(R(1/2,\ba,\bn)\cap\lcm(\bn)),
\end{equation*}
where $\lcm(\bn)=\lcm(n_1, n_2, \dots, n_\ell)$. Since
$R(1/2,\ba,\bn)$ is periodic, one may replace $\lcm(\bn)$ with any
common multiple of $n_1, \dots, n_\ell$ in the above identity.

The proof of~\th~\ref{prescription} relies crucially on the following
extension of Rogers's result to the setting of real residue systems.

\begin{theorem}[Rogers's theorem for real residues]\th\label{thm:rogersext}
  Let $\bn\in\NN^\ell$ and $\eps = 2^{-k}$ for some $k\in\NN$. Then
  \begin{equation}\label{eq:minimum}
    \mu\parens*{R(\eps,\ba, \bn)\cap [0,\lcm(\bn)]} \geq \mu\parens*{R(\eps,\bzero, \bn)\cap [0,\lcm(\bn)]}
  \end{equation}
  for all $\ba\in\RR^\ell$.
\end{theorem}

\begin{remark*}
  When $\ba\in\ZZ^\ell$ and $\eps=1/2$, \th~\ref{thm:rogersext} is a
  restatement of the above mentioned theorem of Rogers. 
\end{remark*}

The functions $\psi$ that we construct in the proof
of~\th~\ref{prescription} give rise to a series of real residue
systems, and the measure bound in~\th~\ref{thm:rogersext} will redound
to lower bounds on the measures of certain approximation sets
associated to $\psi$ and an inhomogeneous parameter $z\in
Z$. \th~\ref{thm:rogersext} allows for a comparison to the homogeneous
case $z=0$. Once there, the measure bounds will come from a sieving
argument that is made easier if there is a uniform upper bound on
$\gcd(n_1, \dots, n_\ell)$, where $n_1, \dots, n_\ell$ are the moduli
of the associated real residue system and lie in a Bohr set.

This motivates us to construct $\psi$ in such a way that the residue
systems that arise have moduli lying in $\PP$, the primes. In order to
show that $\psi$, so constructed, has the desired properties
for~\th~\ref{prescription}, it becomes important to understand the
distribution of primes lying in Bohr sets.

\subsection{Distribution of primes lying in Bohr sets}

We generalize two classical results about primes lying in arithmentic
progressions to the setting of primes lying in Bohr sets.

The first is a generalization of the prime number
theorem for arithmetic
progressions. For $c\in\NN$ and
$a\in\ZZ$, denote
\begin{equation*}
  \PP_a = \set{p\in \PP : p\equiv a\pmod c}.
\end{equation*}
The prime number theorem for arithmetic progressions states that
\begin{equation*}
  \#(\PP_a \cap [1,X]) \sim \frac{X}{\varphi(c)\log X}\qquad (X\to\infty)
\end{equation*}
if $\gcd(a,c)=1$, where $\varphi$ is Euler's totient function. This
was first proved by Siegel--Walfisz. Previously, Dirichlet had shown
that $\sum_{p\in \PP_a} 1/p$ diverges, a result whose corresponding
generalization to Bohr sets we will also need.
(See~\cite[Chapter~II.8]{Tenenbaum})

Let $\bw,\bv\in\RR^\ell$ and $\eps>0$, and denote
\begin{equation*}
  H=\overline{\set{n\bw \bmod 1 : n\in\NN}}\subset \TT^\ell 
\end{equation*}
and its normalized Haar measure $\mu_H$. Let $c$ be the number of
connected components of $H$ and put
\begin{equation*}
  H^\times=\overline{\set{n\bw \bmod 1 : \gcd(n,c)=1}}.
\end{equation*}
Let $\mu^\times$ denote the renormalized restriction of $\mu_H$ to
$H^\times$, that is,
\begin{equation*}
  \mu^\times = \frac{1}{\mu(H^\times)} \bone_{H^\times}\mu_H.
\end{equation*}
We prove the following.

\begin{theorem}[Prime number theorem for Bohr sets]\th\label{DirichletPNT}
  Let $\bw, \bv\in\RR^\ell$ and $\eps>0$. Let $H^\times$ and
  $\mu^\times$ be as above, and let $B(\bv,\eps)$ the ball of radius
  $\eps>0$ centered at $\bv$.
  \begin{enumerate}[label=\alph*)]
  \item If $B(\bv,\eps)\cap H^\times \neq \emptyset$, then
    $\#(\PP\cap N_\bv(\bw, \eps))=\infty$ and, moreover,
    \begin{equation*}
      \#(\PP\cap N_{\bv}(\bw, \eps) \cap [1,X]) \sim \mu^\times(B(\bv,\eps)) \frac{X}{\log X}.
    \end{equation*}
    
  \item If $B(\bv,\eps)\cap H^\times = \emptyset$, then  $\#(\PP\cap N_\bv(\bw, \eps))<\infty$.
  \end{enumerate}
\end{theorem}

\begin{remark*}
  When $w = 1/c$ and $v=aw$, where $(a,c)\in \ZZ\times \NN$, and
  $0 < \eps < 1/c$, the above theorem returns the prime number theorem
  for arithmetic progressions. In this case, $\mu^\times$ is the
  measure supported uniformly on the points
  $\set{n/c : \gcd(n,c)=1}\subset\RR/\ZZ$, and one has
  $B(v,\eps) = \set{a/c}$ and $\mu^\times(B(v,\eps)) = 1/\varphi(c)$.
\end{remark*}

The next result concerns equidistribution. A sequence
$(x_n)_{n\in\NN}$ of real numbers is said to be equidistributed modulo
one if for every nonempty open interval $I\subset \RR/\ZZ$, one has
\begin{equation*}
 \lim_{N\to\infty}\frac{1}{N}\sum_{n\leq N} \bone_I(x_n\bmod 1) =\mu(I).
\end{equation*}
The foundational theorem, known as the Kronecker--Weyl theorem, says
that for $z\in\RR$, the sequence $(nz)_{n\in\NN}$ is equidistributed
modulo one if and only if $z$ is irrational.

This descends to the primes. The sequence $(pz)_{p\in\PP}$ is
equidistributed modulo one if $z$ is irrational, a result that is
often attributed to Vinogradov~\cite{Vinogradov} and can also be found
in work of Rhin~\cite{Rhin} and Vaughan~\cite{Vaughan}. It is not hard
to extend this to primes lying in arithmetic progressions: If
$\gcd(a,c)=1$ and $z$ is irrational, then $(pz)_{p\in\PP_a}$ is
equidistributed modulo one. Again, we generalize from primes lying in
arithmetic progressions to primes lying in Bohr sets.

\begin{theorem}[Equidistribution along Bohr-set primes]\th\label{thm:pbohrequid}
  Let $\by,\bv\in \RR^k$ and $\eps>0$, and let $c$ be the number of connected components of
  \begin{equation*}
    \overline{\set{n\by\bmod 1 : n\in\NN}}.
  \end{equation*}
  If
  \begin{equation*}
    B(\bv,\eps)\cap \overline{\set{n\by\bmod 1 : \gcd(n,c)=1}}\neq \emptyset,
  \end{equation*}
  then for every $z \in \RR\setminus\Span_\QQ\set{1, y_1, \dots, y_k}$,
  the sequence $\parens*{pz}_{p\in \PP\cap N_\bv(\by, \eps)}$ is
  equidistributed modulo one.
\end{theorem}

\begin{remark*}
  When $y = 1/c$, $v=ay$, where $\gcd(a,c)=1$, and $0 < \eps < 1/c$,
  the above theorem says that $(pz)_{p\in \PP_a}$ is equidistributed
  modulo one.
\end{remark*}

\subsection{Further discussion and questions}\label{sec:furth-disc-quest}

Before proceeding to the proofs, we pose some questions that arise
naturally but are not pursued here.

\

The functions $\psi$ that we construct for~\th~\ref{prescription} are
far from monotonic. They are not even monotonic on their support. The
same is true of the classical Duffin--Schaeffer counterexamples
in~\cite{duffinschaeffer} and the inhomogeneous versions
in~\cite{ds_counterex}. The recent paper~\cite{CHPR} constructs
examples that \emph{are} monotonic on their support and in fact can
have values on the support that match any prescribed decreasing
function $f$, although in the inhomogeneous cases they require
restrictions on the inhomogeneous parameters. We ask the following. 

\begin{problem}\th\label{monotoniconsupport}
  Is~\th~\ref{prescription} true with the added condition that $\psi$
  is monotonic on its support? Given a non-increasing function
  $f:\NN \to \RR_{\geq 0}$, is~\th~\ref{prescription} true for a function
  $\psi = \bone_S f$ for some $S\subset\NN$?
\end{problem}

We hasten to point out that the answers cannot simply be yes. By a
result of Yu~\cite[Theorem~1.3]{Yu}, if
$\psi(q) \ll \frac{1}{q(\log\log q)^2}$ and $\sum\psi(q)$ diverges,
then $\mu(W(\psi,y)) = 1$ for every ``tamely Liouville'' number
$y$. This immediately implies that if
$f(q) = \frac{1}{q(\log\log q)^2}$, then~\th~\ref{monotoniconsupport}
must require restrictions on the set $Y\subset\RR$.

\

Regarding~\th~\ref{thm:rogersext}, part of the proof relies on a trick
where we dilate real residue systems $R(\eps,\ba,\bn)$ until their
intervals have unit length. For this reason, it is convenient to
restrict to values of $\eps>0$ that are powers of $1/2$. We ask the
following question.

\begin{problem}\th\label{prob:generalrogers}
  Is~\th~\ref{thm:rogersext} true with general $\eps>0$? 
\end{problem}

One can also ask a version of the question where $\bn\in\RR^\ell$. In
this generality, the residue system $R(\eps,\ba,\bn)$ is not
necessarily periodic, so rather than intersect with a fixed interval
(like $[0,\lcm(\bn)]$), it would make sense to intersect with an
interval $[-X,X]$ and study measure bounds as $X\to\infty$.

\

Finally, we pose an additive combinatorial question, whose answer
appears in Appendix~\ref{sec:manuel-hauke}, contributed by Manuel
Hauke. In Lemma~\ref{lem:1/p} we show that the Bohr sets arising in
our constructions contain enough primes that the sums of their
reciprocals diverge. The question arose whether this phenomenon holds
more generally.

\begin{problem}[Answered by~\th~\ref{neg_answer}]\th\label{ntproblem}
  Does every positive-density set $A\subset\NN$ have a subset
  $B\subset A$ such that $\sup_{m,n\in B}\gcd(m,n) < \infty$ and such
  that $\sum_{n\in B}\frac{1}{n}$ diverges?
\end{problem}

After the first posting of this paper, Manuel Hauke wrote to us and
kindly gave a negative answer. \th~\ref{thm_charact} shows that the sets $A$
for which the answer is ``yes'' are precisely those that contain
$n_0\calP$ for some $n_0\in\NN$ and $\calP\subset\PP$ with
$\sum_{p\in\calP}1/p$ diverging. \th~\ref{neg_answer} uses this
characterization to construct density-one sets for
which~\th~\ref{ntproblem} does not hold.

\section{Proof of Theorem~\ref{prescription}}

For $q\in\NN$, $y\in\RR$, and $\delta>0$, denote
\begin{equation*}
  A_q^y(\delta) = \set{x\in[0,1] : \norm{qx-y} < \delta}. 
\end{equation*}
When $\psi:\NN\to\RR$, write $A_q^y(\psi) = A_q^y(\psi(q))$ for brevity.

The proof of~\th~\ref{prescription} reduces to the following
proposition, whose proof appears in Section~\ref{sec:proof-th-refm}.

\begin{proposition}\th\label{modestmeasureofunions}
  Let $\set{y_1, y_2, \dots, y_k}\subset\RR$ and
  $z\in \RR\setminus \Span_\QQ\set{1, y_1, \dots, y_k}$. Let
  $0 < \eps < 1/2$. Then there exists a square-free integer $P>1$ such
  that
  \begin{equation}\label{eq:5}
    \mu\parens*{\bigcup_{\substack{p\mid P \\ p\textrm{ prime}}} A_{P/p}^y\parens*{\frac{\eps}{2p}}} < 2\eps\quad\textrm{for each}\quad y\in \set{y_1, \dots, y_k},
  \end{equation}
  and
  \begin{equation}\label{eq:6}
    \mu\parens*{\bigcup_{\substack{p\mid P \\ p\textrm{ prime}}} A_{P/p}^z\parens*{\frac{\eps}{2p}}} \geq \frac{1}{3}.
  \end{equation}
\end{proposition}

\begin{proof}[Proof of~\th~\ref{prescription} modulo~\th~\ref{modestmeasureofunions}]
  The function $\psi$ will be supported in a union of blocks which are
  most conveniently enumerated as a triangular array,
  \begin{gather*}
    S_{1,1}, \\
    S_{2,1},\quad S_{2,2}, \\
    S_{3,1},\quad S_{3,2},\quad S_{3,3}, \\
    \dots
  \end{gather*}
  so that
  $\supp\psi = \bigcup_{k=1}^\infty\bigcup_{\ell=1}^k S_{k,\ell}$
  where each $S_{k,\ell}$ is a finite set of integers which we now
  define.

  First, enumerate the sets $Y=\set{y_1, y_2, y_3, \dots}$ and
  $Z = \set{z_1, z_2, z_3, \dots}$. We proceed inductively, moving
  through $\set{(k,\ell)\in\ZZ_+^2 : 1\leq \ell \leq k}$ according to the lexicographic ordering
  \begin{align*}
    (i,j) \prec (k,\ell)
    &\iff [i < k] \textrm{ or } [i = k \textrm{ and } j < \ell] \\
    &\iff \lex(i,j) < \lex(k,\ell)
  \end{align*}
  where
  \begin{equation*}
    \lex(k,\ell) = \frac{k(k-1)}{2} + \ell. 
  \end{equation*}
  For each $k,\ell$, we
  apply~\th~\ref{modestmeasureofunions} with
  $y_1, \dots, y_k, z_\ell, 2^{-k}$ filling the roles of
  $y_1, \dots, y_k, z, \eps$, respectively. This results in a
  square-free $P_{k,\ell}>1$ such that~(\ref{eq:5}) and~(\ref{eq:6})
  hold. Set
  \begin{equation*}
    S_{k,\ell} = \set*{\frac{2^{m_{k,\ell}} P_{k,\ell}}{p} : p\mid P_{k,\ell}\textrm{ and } p \textrm{ is prime}},
  \end{equation*}
  where $m_{k,\ell}$ is a sufficiently large natural number (to be
  chosen), and define $\psi:\NN\to\RR$ by
  \begin{equation*}
    \psi(q) = 
    \begin{cases}
      \displaystyle\frac{2^{-k}}{2p} &\textrm{if } q\in S_{k,\ell} \textrm{ for some } k, \ell \\
      \displaystyle 0 &\textrm{otherwise.}
    \end{cases}
  \end{equation*}
  Notice that for arbitrary $w\in \RR$,  we have
  \begin{equation*}
    \bigcup_{q\in S_{k,\ell}} A_q^{w}(\psi) = T^{-m_{k,\ell}} \parens*{\bigcup_{p \mid P_{k,\ell}} A_{P_{k,\ell}/p}^{w}\parens*{\frac{2^{-k}}{2p}}},
  \end{equation*}
  where $T:\TT\to\TT$ is the circle doubling map.  Since it is measure
  preserving we have
  \begin{equation}\label{eq:3}
    \mu\parens*{\bigcup_{q \in S_{k,\ell}} A_q^y(\psi)} < 2^{-k}\quad\textrm{for each}\quad y\in \set{y_1, \dots, y_k},
  \end{equation}
  and
  \begin{equation}\label{eq:4}
    \mu\parens*{\bigcup_{q\in S_{k,\ell}} A_q^{z_\ell}(\psi)} \geq \frac{1}{3},
  \end{equation}
  by~(\ref{eq:5}) and~(\ref{eq:6}).  Furthermore, $T$ is mixing,
  meaning in particular that
  \begin{multline*}
    \lim_{m\to\infty}\mu\parens*{\bigcup_{q\in S_{i,j}}A_q^z(\psi)\cap T^{-m}\parens*{\bigcup_{p \mid P_{k,\ell}}A_{P_{k,\ell}/p}^z\parens*{\frac{2^{-k}}{2p}}}} \\
    = \mu\parens*{\bigcup_{q\in S_{i,j}}A_q^z(\psi)}\mu\parens*{\bigcup_{p \mid P_{k,\ell}}A_{P_{k,\ell}/p}^z\parens*{\frac{2^{-k}}{2p}}} \\
    =\mu\parens*{\bigcup_{q\in S_{i,j}}A_q^z(\psi)} \mu\parens*{\bigcup_{q\in S_{k,\ell}}A_q^z(\psi)}.
  \end{multline*}
  Hence, the parameter $m_{k,\ell}$ can be chosen large enough that
  \begin{equation*}
    \min S_{k,\ell} > \max S_{i,j}
  \end{equation*}
  and
  \begin{equation}\label{eq:quasiind}
    \mu\parens*{\bigcup_{q\in S_{i,j}}A_q^z(\psi)\cap\bigcup_{q\in S_{k,\ell}}A_q^z(\psi)} \leq \parens*{1 + 2^{-k}}\mu\parens*{\bigcup_{q\in S_{k,\ell}}A_q^z(\psi)} \mu\parens*{\bigcup_{q\in S_{i,j}}A_q^z(\psi)}
  \end{equation}
  for all $(i,j) \prec (k,\ell)$. In particular, one may take
  $m_{1,1}=0$.

  \

  We now verify that $\psi$ fulfills the claims made
  in~\th~\ref{prescription}. Let $y\in Y$. Then $y=y_{k'}$ for some
  $k'\in\NN$. By~(\ref{eq:3}), we have
  \begin{equation*}
    \mu\parens*{\bigcup_{q\in S_{k,\ell}} A_q^y(\psi)} < 2^{-k}
  \end{equation*}
  for all $k\geq k'$, which implies that 
  \begin{equation*}
    \sum_{k=1}^\infty\sum_{\ell = 1}^k \mu\parens*{\bigcup_{q\in S_{k,\ell}} A_q^y(\psi)} < \infty. 
  \end{equation*}
  Now by the Borel--Cantelli lemma, 
  \begin{align*}
    \mu\parens*{\limsup_{\lex(k,\ell)\to \infty} \bigcup_{q\in S_{k,\ell}} A_q^y(\psi)} = 0.
  \end{align*}
  But
  \begin{equation*}
    \limsup_{\lex(k,\ell)\to \infty} \bigcup_{q\in S_{k,\ell}} A_q^y(\psi) = \limsup_{q\to \infty} A_q^y (\psi)  =   W(\psi, y),
  \end{equation*}
  therefore, $\mu(W(\psi, y))=0$, proving the first part of~\th~\ref{prescription}. 

  Let $z\in Z$. Then $z = z_{\ell'}$ for some $\ell'\in\NN$,
  and~(\ref{eq:4}) shows that we have
  \begin{equation}
    \mu\parens*{\bigcup_{q\in S_{k,\ell'}} A_q^z(\psi)} \geq \frac{1}{3}
  \end{equation}
  for all $k\geq \ell'$. By~(\ref{eq:quasiind}) and the Erd{\H
    o}s--Tur{\'a}n Lemma~\cite{ErdosRenyiBC}, we have for all
  $k_0\geq \ell'$ that
  \begin{align*}
    \mu\parens*{\limsup_{k\to \infty} \bigcup_{q\in S_{k,\ell'}} A_q^z(\psi)} \geq \frac{1}{1 + 2^{-k_0}}. 
  \end{align*}
  Since
  \begin{equation*}
    \limsup_{k\to \infty} \bigcup_{q\in S_{k,\ell'}} A_q^z(\psi) \subset \limsup_{q\to \infty} A_q^z (\psi)  =   W(\psi, z),
  \end{equation*}
  we have $\mu(W(\psi, z))\geq (1 + 2^{-k_0})^{-1}$.  Sending
  $k_0\to\infty$ gives $\mu(W(\psi, z)) = 1$.  This finishes the proof
  of~\th~\ref{prescription}.
\end{proof}

\section{Real residue systems}

We now turn our attention to real residue systems $R(\eps,\ba,\bn)$,
where $\eps>0$, $\bn\in\NN^\ell$, and $\ba\in\RR^\ell$, and to the proof of~\th~\ref{thm:rogersext}.

\begin{proof}[Proof of~\th~\ref{thm:rogersext}]
  Let $\eps$ and $\bn$ be as in the theorem's statement. Note that
  \begin{equation*}
      \mu(\eps,\ba,\bn) := \mu\parens*{R(\eps,\ba, \bn)\cap [0,N]},
  \end{equation*}
  where $N=\lcm(\bn)$, is a continuous function of $\ba$. Therefore,
  it suffices to establish~\eqref{eq:minimum} for all $\ba$ with
  dyadic rational coordinates.

  \

  We shall now argue by induction that
  $\mu(\frac{1}{2},\ba, \bn)\geq \mu(\frac{1}{2}, \bzero, \bn)$ holds
  for all $\ba\in 2^{-j}\ZZ^\ell$ with $j\geq 0$. Rogers's
  theorem~\cite[Page~242]{HalberstamRoth} gives the base case, namely,
  that $\mu(\frac{1}{2},\ba,\bn)\geq \mu(\frac{1}{2},\bzero, \bn)$ for
  all $\ba \in \ZZ^\ell$.
  
   Suppose it is known that
   $\mu(\frac{1}{2},\ba', \bn)\geq \mu(\frac{1}{2},\bzero,\bn)$ for all
   $\ba' \in 2^{-j}\ZZ^\ell$. Let $\ba$ be some fixed element of
   $2^{-j-1}\ZZ^\ell\setminus 2^{-j}\ZZ^\ell$, and let
   \begin{equation*}
     I = \set{i : a_i2^{j+1} \textrm{ is odd}}.
   \end{equation*}
   Then $I$ is nonempty.
    
   Let $e_I = \sum_{i\in I} e_i,$ where $e_i$ is the $i$-th standard
   unit vector. Then
   \begin{equation*}
     \ba - 2^{-j-1}e_I \in 2^{-j}\ZZ^\ell \qquad\textrm{and}\qquad \ba +2^{-j-1}e_I\in 2^{-j}\ZZ^\ell.
   \end{equation*}
   We will demonstrate that $\mu(\frac{1}{2}, \ba + te_I, \bn)$ is linear
   in $t$ over the interval $[-2^{-j-1},2^{-j-1}].$ It will follow by
   the induction hypothesis that
   \begin{equation*}
     \mu\parens*{\frac{1}{2},\ba, \bn} \geq \min\set*{\mu\parens*{\frac{1}{2},\ba-2^{-j-1}e_I,\bn}, \mu\parens*{\frac{1}{2},\ba+2^{-j-1}e_I, \bn}}\geq \mu\parens*{\frac{1}{2},\bzero, \bn}.
   \end{equation*}

   \

   Let 
   \begin{equation*}
     E:=\bigcup_{i\in I} \parens*{a_i+ \brackets*{-\frac{1}{2}, \frac{1}{2}} + n_i \ZZ},
   \end{equation*}
   and
   \begin{equation*}
     F:=\bigcup_{i\in \set{1, \dots, \ell}\setminus I} \parens*{a_i+ \brackets*{-\frac{1}{2}, \frac{1}{2}} + n_i \ZZ}.
   \end{equation*}
   Then
   \begin{multline*}
       \mu(1/2, \ba+te_I, \bn) = \mu(((E+t)\cup F)\cap [0,N]) \\
       = \mu((E+t) \cap [0,N]) + \mu(F\cap [0,N]) - \mu((E+t)\cap F\cap [0,N]).
   \end{multline*}
   Since $E$ is periodic with period $N,$ $\mu((E+t) \cap [0,N])$ is
   constant in $t.$ Thus, to establish the linearity of
   $\mu(1/2, \ba+te_I, \bn)$ over $[-2^{-j-1},2^{-j-1}]$ we need only
   consider how $\mu((E+t)\cap F\cap [0,N])$ changes with respect to
   $t.$

   Write
   \begin{equation*}
     E=\bigsqcup_{I\in\calI} I \qquad\textrm{and}\qquad F=\bigsqcup_{J\in\calJ} J
   \end{equation*}
   where $\calI$ and $\calJ$ are families of disjoint intervals having
   endpoints in $2^{-j-1}\ZZ\setminus 2^{-j}\ZZ$ and $2^{-j}\ZZ$,
   respectively. Then
   \begin{equation*}
     \mu((E+t)\cap F \cap [0,N]) = \bigsqcup_{\substack{I\in\calI \\ J \in \calJ}} \mu((I+t)\cap J\cap [0,N]).
   \end{equation*}
   For each $I,J$, $\mu((I+t)\cap J)$ is piecewise linear having 
   corners lying in $2^{-j-1}\ZZ$. It follows that each is linear in
   $[-2^{-j-1}, 2^{-j-1}]$, hence so are
   $\mu((E+t)\cap F\cap [0,N])$ and $\mu(((E+t)\cup F)\cap [0,N])$
   as claimed. This shows $\mu(1/2, \ba, \bn) \geq \mu(1/2, \bzero, \bn)$ for
   all $\ba\in 2^{-j-1}\ZZ^\ell$. By induction, the same is true for
   all $\ba$ having dyadic coordinates, and by continuity of $\mu$, we
   have $\mu(1/2,\ba, \bn)\geq \mu(1/2, \bzero, \bn)$ for all $\ba\in\RR^\ell$.

   Suppose now that $\eps = 2^{-k}$ for some $k\geq 2$, and
   $\ba\in\RR^\ell$. Then
   \begin{align*}
     \mu(\eps,\ba,\bn)
     &= \frac{1}{2^{k-1}}\mu\parens*{1/2,2^{k-1}\ba, 2^{k-1}\bn} \\
     &\geq \frac{1}{2^{k-1}}\mu\parens*{1/2,\bzero, 2^{k-1}\bn} \\
     &= \mu(\eps,\bzero, \bn),
   \end{align*}
   completing the proof.
 \end{proof}

 \section{Distribution of primes lying in Bohr sets}

A sequence $(x_n)_{n\in\NN}$ in a compact topological group $G$ is
said to be \emph{equidistributed} with respect to a Borel measure
$\mu$ if for every nonempty open set $I\subset G$, one has
\begin{equation*}
 \lim_{N\to\infty}\frac{1}{N}\sum_{n\leq N} \bone_I(x_n) =\mu(I).
\end{equation*}
One of the standard methods for establishing the equidistribution of
sequences is Weyl's criterion.

\begin{theorem}[{Weyl's criterion,~\cite[Corollary~5.2]{KuipersNiederreiter}}]
  Let $G$ be a compact abelian group. A sequence $x_n\in G$ is
  equidistributed with respect to Haar measure $\mu_G$ if and only if
  for every non-trivial periodic character $\chi:G\to \CC^\times$, we
  have
  \begin{equation*}
    \sum_{n\leq N}\chi(x_n) = o(N) 
  \end{equation*}
  as $N\to \infty$
\end{theorem}

\subsection{Equidistribution of toral translations}

It is well-known that if the real numbers $1, w_1, w_2, \dots w_\ell$
are linearly independent over $\QQ$, then the orbit
$\parens{n\bw \bmod 1}_{n\in\NN}$, where $\bw = (w_1, \dots, w_\ell)$,
equidistributes in the torus $\TT^\ell$ with respect to Lebesgue
measure, $\mu:=\mu_{\TT^\ell}$. More generally, the orbit of an
arbitrary vector $\bw\in\RR^\ell$ equidistributes in its orbit
closure,
\begin{equation}\label{eq:orbitclosure}
  H:=\overline{\set{n\bw\bmod 1 : n\in\NN}},
\end{equation}
which is a subgroup of $\TT^\ell$. Here, the equidistribution is with
respect to Haar measure, $\mu_H$, on $H$. This is known as the
Kronecker--Weyl theorem, and it can be proved by applying Weyl's
criterion.

Furthermore, the orbit closure $H$ is a union of disjoint affine
sub-tori,
\begin{equation}\label{eq:Hsubtori}
  H = H_0 \sqcup H_1 \sqcup \dots \sqcup H_{c-1},
\end{equation}
where $H_a = a\bw + H_0$, for $a=0, 1, \dots, c-1$ and $c$ is the size
of the torsion subgroup of $\TT^\ell/H$. The orbit of $\bw$ cycles
through the tori, with
\begin{equation}
  \label{eq:cycle}
  n\bw \bmod 1 \in H_{n \bmod c} \quad (n\in\NN).
\end{equation}
It follows, again via Weyl's criterion, that
$\parens{n\bw \bmod 1}_{n\equiv 0 \pmod c}$ is equidistributed in
$H_0$ with respect to normalized Haar measure
$\mu_0:=\mu_{H_0}$. Consequently,
$\parens{n\bw \bmod 1}_{n\equiv a \pmod c}$ is equidistributed in
$H_a$ with respect to the pushforward measure
\begin{equation}\label{eq:pushforward}
  \mu_a = T_*^a\mu_{H_0},
\end{equation}
where $T^a:H_0\to H_a$ is the mapping $\bx \mapsto \bx + a\bw \pmod 1$.

\subsection{Equidistribution along prime times}

The proof of the following result is usually attributed to
Vinogradov~\cite{Vinogradov} and can be found in work of
Vaughan~\cite{Vaughan} or Rhin~\cite{Rhin}.

\begin{theorem}\th\label{pyequidistribution}
  If $z\in\RR\setminus\QQ$, then $(pz \bmod 1)_{p \textrm{ prime}}$
  is equidistributed in $\TT = \RR/\ZZ$ with respect to Lebesgue measure.
\end{theorem}

\th~\ref{pyequidistribution} is easily generalized to higher
dimensions: $\parens{p\bw\bmod 1}_{p\in\PP}$ is equidistributed in
$\TT^\ell$ if $\parens{n\bw\bmod 1}_{n\in\NN}$ is equidistributed in
$\TT^\ell$, that is, if $1, w_1, w_2, \dots, w_\ell$ are linearly
independent over $\QQ$. But extra care must be taken in the further
generalization to all $\bw\in\RR^\ell$. If $\gcd(a,c)>1$, then the set
$\set{n\in\NN : n\equiv a\pmod c}$ has at most one prime member,
therefore, by~(\ref{eq:cycle}), $\parens{p\bw\bmod 1}_{p\in\PP}$
cannot in general equidistribute---some subtori may be missed
entirely. On the other hand, if $\gcd(a,c)=1$, then by the prime
number theorem for arithmetic progressions, we have
\begin{equation}
  \label{eq:dirichletPNT}
  \#(\PP_a \cap [1,X]) \sim \frac{X}{\varphi(c)\log X}
\end{equation}
where
\begin{equation*}
  \PP_a := \set{p\in\PP : p\equiv a\pmod c}.
\end{equation*}
In the following proposition, we show that if $\gcd(a,c)=1$, then
$\parens{n\bw\bmod 1}$ equidistributes in $H_a$ along the prime
members of $\set{n\in\NN : n\equiv a\pmod c}$.

\begin{proposition}\th\label{prop:subtorusequid}
  Let $\bw\in \RR^\ell$ and $c$ the number of connected components of
  $\overline{\set{n\bw\bmod 1 : n\in\NN}}\subset\TT^\ell$. If
  $\gcd(a,c)=1$, then $\parens{p\bw\bmod 1}_{p\in\PP_a}$ is
  equidistributed with respect to $\mu_a$, the measure given
  by~(\ref{eq:pushforward}).
\end{proposition}

\begin{proof}
  The claim is equivalent to $\parens{(p-a)\bw\bmod 1}_{p\in \PP_a}$
  equidistributing with respect to $\mu_0$. By Weyl's criterion, this
  will follow once it is shown that for every nontrivial character
  $\chi:H_0 \to \CC$, we have
  \begin{equation}\label{eq:9}
    \sum_{\substack{p\leq X \\ p\equiv a\pmod c}} \chi((p-a) \bw) = o(\#\parens*{(\PP_a-a)\cap[1,X]}).
  \end{equation}
  A character $\chi$ is non-trivial if and only if $\chi(H_0)=\TT$,
  which, since $\set{n\bw\bmod 1 : n\in\NN}\cap H_0$ is dense in
  $H_0$, is equivalent to
  $\overline{\set{\chi(n\bw) : n\in\NN}} = \TT$. Such characters are
  of the form $\bx\mapsto e(\inner{\bh, \bx})$ where $\bh\in \ZZ^\ell$
  and $\inner{\bh,\bw}$ is irrational. Now
  \begin{align*}
    \sum_{\substack{p\leq X \\ p\equiv a\pmod c}} \chi((p-a) \bw)
    &=\sum_{\substack{p\leq X \\ p\equiv a\pmod c}} e(\inner{\bh,(p-a)\bw})\\
    &=\sum_{p\leq X }\bone_{\set{a\pmod c}}(p) e((p-a)\inner{\bh,\bw})\\
    &=\sum_{p\leq X }\sum_{k=0}^ce\parens*{(p-a)\frac{k}{c}} e((p-a)\inner{\bh,\bw})\\
    &=\sum_{k=0}^c e\parens*{-a\parens*{\frac{k}{c} + \inner{\bh,\bw}}}\underbrace{\sum_{p\leq X }e\parens*{p\parens*{\frac{k}{c}+\inner{\bh,\bw}}}}.
  \end{align*}
  Since $k/c + \inner{\bh,\bw}$ is
  irrational,~\th\ref{pyequidistribution} and Weyl's criterion imply
  that the underbraced sum is $o(\pi(X))$, where
  \begin{equation*}
    \pi(X) = \#(\PP\cap [1,X]). 
  \end{equation*}
  Therefore,
  \begin{align*}
    \sum_{\substack{p\leq X \\ p\equiv a\pmod c}} \chi((p-a) \bw) = o(\pi(X)).
  \end{align*}
  Finally, by the Siegel--Walfisz prime number theorem for arithmetic
  progressions, we have
  \begin{equation*}
    \#\parens*{(\PP_a-a)\cap[1,X]} \sim \frac{1}{\varphi(c)}\pi(X),
  \end{equation*}
  so~(\ref{eq:9}) follows and the proof is finished.
\end{proof}

\subsection{The prime number theorem for Bohr sets}

We can now state the proof of~\th~\ref{DirichletPNT}. 

\begin{proof}[Proof of~\th~\ref{DirichletPNT}]
  Let $\bw,\bv\in\RR^\ell$ and $\eps>0$. Note that the renormalized
  restriction of $\mu_H$ to
  \begin{equation*}
    H^\times = \overline{\set{n\bw\bmod 1 : \gcd(n,c)=1}}
  \end{equation*}
  is given by
  \begin{equation}\label{eq:mutimes}
    \mu^\times = \frac{1}{\varphi(c)}\sum_{\substack{a=0 \\ \gcd(a,c)=1}}^{c-1}\mu_a. 
  \end{equation}
  Now
  \begin{multline*}
    \#(\PP\cap N_\bv(\bw,\eps)\cap [1,X])
    = \sum_{\substack{p\leq X \\ p \textrm{ prime}}} \bone_{B(\bv,\eps)}(p\bw\bmod 1) \\
    = \sum_{\substack{a=0 \\ \gcd(a,c)=1}}^{c-1}\sum_{p\in \PP_a\cap [1,X]} \bone_{B(\bv,\eps)}(p\bw\bmod 1)\\
     + \underbrace{\sum_{\substack{a=0 \\ \gcd(a,c)>1}}^{c-1}\sum_{p\in \PP_a\cap [1,X]} \bone_{B(\bv,\eps)}(p\bw\bmod 1)}_{O(1)}.
  \end{multline*}
  If $B(\bv,\eps)\cap H^\times \neq \emptyset$, then
  $\mu^\times(B(\bv,\eps))>0$ and by~\th~\ref{prop:subtorusequid} the
  above gives
    \begin{align*}
    \#(\PP\cap N_\bv(\bw,\eps)\cap [1,X])
    &\sim \sum_{\substack{a=0 \\ \gcd(a,c)=1}}^{c-1}\#(\PP_a\cap [1,X])\mu_a(B(\bv,\eps)) \\
    &\sim \mu^\times(B(\bv, \eps))\frac{X}{\log X},
  \end{align*}
  by the prime number theorem for arithmetic
  progressions~(\ref{eq:dirichletPNT}) and~(\ref{eq:mutimes}).

  On the other hand, if $B(\bv,\eps)\cap H^\times = \emptyset$, then
  \begin{align*}
    \#(\PP\cap N_\bv(\bw,\eps)\cap [1,X])
    &= \sum_{\substack{a=0 \\ \gcd(a,c)>1}}^{c-1}\sum_{p\in \PP_a\cap [1,X]} \bone_{B(\bv,\eps)}(p\bw\bmod 1) \\
    &<\infty. 
  \end{align*}
  This proves the theorem.
\end{proof}

\subsection{Equidistribution along primes in Bohr sets}

One can use \th~\ref{prop:subtorusequid} to quickly
generalize~\th~\ref{pyequidistribution} to primes lying in arithmetic
progressions $\set{n\in \NN : n\equiv a\pmod c}$ with
$\gcd(a,c)=1$. Given $z\in\RR\setminus\QQ$, put $\bw=(y,z)$, where
$y=1/c$, so that
\begin{equation}\label{eq:subtori}
  \overline{\set{n\bw\bmod 1 : n\in\NN}} = H_0\sqcup \dots\sqcup H_{c-1},
\end{equation}
and
\begin{equation*}
  H_a = \set{ay \bmod 1} \times \TT = \set*{a/c} \times \TT \qquad (a=0,\dots, c-1).
\end{equation*}
\th~\ref{prop:subtorusequid} says that
$\parens{p\bw\bmod 1}_{p\in\PP_a}$ is equidistributed in $H_a$, which
now is equivalent to the equidistribution of
$\parens{pz\bmod 1}_{p\in\PP_a}$ in $\TT$.

Notice that the arithmetic progression
$\set{n\in \NN : n\equiv a\pmod c}$ can be seen as the Bohr set
\begin{equation*}
  N_{ay}(y,\eps) = \set{n\in\NN : \norm{ny-a/c} < \eps}
\end{equation*}
for any $0 < \eps < 1/c$. \th~\ref{thm:pbohrequid}, which we can now
prove, generalizes~\th~\ref{pyequidistribution} to primes lying in
Bohr sets $N_\bv(\by,\eps)$. As with the generalization to arithmetic
progressions, it is required that the Bohr set actually contains
infinitely many primes. The irrationality assumption is replaced by
$z\notin\Span_\QQ\set{1,y_1, \dots, y_k}$.

\begin{proof}[Proof of~\th~\ref{thm:pbohrequid}]
  Suppose $\by,\bv\in\RR^k$ and $\eps>0$ are such that
  \begin{equation*}
    B(\bv,\eps)\cap \overline{\set{n\by\bmod 1 : \gcd(n,c)=1}}\neq \emptyset,
  \end{equation*}
  where $c$ is the number of connected components in
  \begin{equation*}
    \overline{\set{n\by\bmod 1 : n\in\NN}} = H_0 \sqcup\dots\sqcup H_{c-1}.
  \end{equation*}
  Let $z\in \RR\setminus\Span_\QQ\set{1, y_1, \dots, y_k}$. It is claimed
  that $\parens{pz\bmod 1}_{p\in \PP\cap N_\bv(\by,\eps)}$ is
  equidistributed in $\TT$.

  Let $I\subset\TT$ be an interval. Put $\bw = (\by, z)$. Since
  $z\notin \Span_\QQ\set{1, y_1, \dots, y_k}$, we have
  \begin{equation*}
    \overline{\set{n\bw\bmod 1 : n\in \NN}} = \overline{\set{n\by\bmod 1 : n\in \NN}} \times \TT,
  \end{equation*}
  and~(\ref{eq:subtori}) may be written as
  \begin{equation*}
    \overline{\set{n\bw\bmod 1 : n\in\NN}} = (H_0\times\TT) \sqcup\dots\sqcup (H_{c-1}\times\TT)
  \end{equation*}
  Therefore,
  \begin{equation*}
    \parens*{B(\bv,\eps)\times I}\cap \overline{\set{n\bw \bmod 1 : n\in\NN}} = \parens*{\bigcup_{a=0}^{c-1}(B(\bv,\eps)\cap H_a)}\times I.
  \end{equation*}
  Note that
  \begin{equation*}
    \PP\cap N_\bv(\by, \eps) = \bigcup_{a=0}^{c-1}(\PP_a \cap N_\bv(\by, \eps)), 
  \end{equation*}
  hence 
  \begin{equation*}
    \#(\PP\cap N_\bv(\by, \eps)\cap[1,X]) = \sum_{a=0}^{c-1}\#(\PP_a \cap N_\bv(\by, \eps)\cap [1,X]).
  \end{equation*}
  Since, by~\th~\ref{DirichletPNT}, $N_\bv(\by,\eps)$ contains infinitely
  many primes, it follows that the set of $a\in \set{0, \dots, c-1}$
  for which $\#\parens*{\PP_a\cap N_\bv(\by,\eps)} = \infty$ is nonempty. For
  all such $a$, one has $\gcd(a,c)=1$, and it follows that
  \begin{equation}\label{eq:count}
    \#(\PP\cap N_\bv(\by, \eps)\cap[1,X]) \sim \sum_{\substack{a=0 \\ \gcd(a,c)=1}}^{c-1}\#(\PP_a \cap N_\bv(\by, \eps)\cap [1,X])
  \end{equation}
  as $X\to\infty$.

  Now by~\th~\ref{prop:subtorusequid}, whenever $\gcd(a,c)=1$ we have
  that $(p\bw\bmod 1)_{p\in\PP_a}$ is equidistributed in
  $H_a \times \TT$ with respect to $\mu_a\times\mu_\TT$, where
  $\mu_a$ is a pushforward measure as
  in~(\ref{eq:pushforward}). Therefore,
  \begin{align}
    \sum_{\substack{p\leq X \\ p\in \PP\cap N_\bv(\by, \eps)}}\bone_I(p z \bmod 1)
    &= \sum_{a=0}^{c-1}\sum_{\substack{p\leq X \\ p\in \PP_a\cap N_\bv(\by, \eps)}}\bone_I(p z \bmod 1) \nonumber\\
    &= \sum_{a=0}^{c-1} \sum_{p\in \PP_a\cap[1,X]}\bone_{\brackets*{B(\bv,\eps)\times I}}(p (\by,z) \bmod 1) \nonumber \\
    &\sim  \sum_{\substack{a=0\\\gcd(a,c)=1}}^{c-1} \#\parens*{\PP_a\cap[1,X]} (\mu_a\times \mu_{\TT})(B(\bv,\eps)\times I) \label{eq:comeback} \\
    &=  \sum_{\substack{a=0\\\gcd(a,c)=1}}^{c-1}\#\parens*{\PP_a\cap[1,X]} \mu_{H_a}(B(\bv, \eps)) \mu_\TT(I) \nonumber \\
    &\sim \mu_\TT(I)\sum_{\substack{p\leq X \\ p\in \PP\cap N_\bv(\by, \eps)}} 1, \nonumber
  \end{align}
  again by~\th~\ref{prop:subtorusequid} and~(\ref{eq:count}). In other words, 
  \begin{equation*}
\sum_{\substack{p\leq X \\ p\in \PP\cap N_\bv(\by, \eps)}}\bone_I(p z \bmod 1)
    \sim \mu_\TT(I)\#(\PP\cap N_\bv(\by, \eps) \cap [1,X])
  \end{equation*}
  as $X\to\infty$, which is the claimed equidistribution.
\end{proof}

The following lemma generalizes a result of Dirichlet. 

\begin{lemma}\th\label{lem:1/p}
  Let $\by\in \RR^k$, $\eps>0$, and let $I\subset \RR/\ZZ$ be an
  interval. Then for $z\in \RR\setminus\Span_\QQ\set{1,y_1, \dots, y_k}$,
  \begin{equation*}
    \sum_{\substack{p\in \PP\cap N_\by(\by,\eps)\\ pz \in I}}\frac{1}{p} = \infty.
  \end{equation*}
\end{lemma}

\begin{proof}
  Let $c$ be the number of connected components of
  \begin{equation*}
    \overline{\set{n\by\bmod 1 : n\in\NN}}.
  \end{equation*}
  Notice that 
  \begin{equation*}
    B(\by,\eps)\cap \overline{\set{n\by\bmod 1 : \gcd(n,c)=1}}\neq \emptyset,
  \end{equation*}
  which implies that $\mu^\times(B(\by,\eps))>0$. Since
  $z \in \RR\setminus\Span_\QQ\set{1, y_1, \dots, y_k}$, we have
  by~\th~\ref{thm:pbohrequid,DirichletPNT} that
  \begin{align}
    \widehat\pi(X)
    &:=\sum_{\substack{p\leq X \\ p\in\PP\cap N_\by(\by,\eps)}}\bone_I(p z\,(\bmod\, 1)) \nonumber\\
   &\sim \#\parens*{\PP\cap N_\by(\by,\eps)\cap [1,X]} \mu(I) \nonumber \\
    &\sim \mu^\times(B(\by,\eps)) \frac{X}{\log X}\mu(I). \label{eq:8}
  \end{align}
  Now
  \begin{align*}
    \sum_{\substack{p\leq X \\ p\in \PP\cap N_\by(\by,\eps)\\ pz \in I}}\frac{1}{p}
    &= \sum_{n\leq X} (\widehat\pi(n) - \widehat\pi(n-1))\frac{1}{n} \\
    &= -\widehat\pi(1)\frac{1}{2} + \sum_{n < X} \widehat\pi(n)\parens*{\frac{1}{n} - \frac{1}{n+1}}  + \frac{\widehat\pi(X)}{X}\\
    &\asymp \sum_{n <  X} \frac{1}{n \log n}
  \end{align*}
  by~(\ref{eq:8}). This diverges as $X\to\infty$.
\end{proof}

\section{Proof of~\th~\ref{modestmeasureofunions}}\label{sec:proof-th-refm}

\begin{lemma}\th\label{lem:smallmeasure}
  Let $\by = (y_1, \dots, y_k)\in \RR^k$. Suppose
  $p\in N_\by(\by, \eps/2)$ and $p|P$. Then for each
  $y\in \set{y_1, \dots, y_k}$,
  \begin{equation}\label{eq:containment}
    A_{q}^y\parens*{\frac{\eps}{2p}} \subset A_{P}^y\parens*{\eps},
  \end{equation}
  where $q = P/p$.
\end{lemma}

\begin{proof}
  The set $A_{q}^y\parens*{\frac{\eps}{2p}}$ is a union of intervals
  of length $\eps/P$, and $A_{P}^y\parens*{\eps}$ is a union of
  intervals of length $2\eps/P$. Therefore, the
  containment~(\ref{eq:containment}) will follow upon showing that
  each constituent interval of $A_{q}^y\parens*{\frac{\eps}{2p}}$ is
  centered at a point that is within $\eps/2P$ of a center of an
  interval from $A_P^y(\eps)$. To that end, let $a\in \ZZ$. The
  distance from $(a+y)/q$ to the set of centers of the constituent
  intervals of $A_P^y(\eps)$ is
  \begin{align*}
    \min_{b\in \ZZ}\abs*{\frac{a+y}{q} - \frac{b+y}{P}}
    &= \frac{1}{P}\min_{b\in \ZZ}\abs{(p-1)y + pa-b} \\
    &= \frac{1}{P}\norm{(p-1)y} < \frac{\eps}{2P}.
  \end{align*}
  The lemma follows.
\end{proof}

\th~\ref{lem:smallmeasure} will be used to establish~(\ref{eq:5}). The
main substance of the proof of~\th~\ref{modestmeasureofunions} is in
showing that for $X>1$ large, the primes $p\leq X$ with
$p\in N_\by(\by, \eps/2)$ satisfy~(\ref{eq:6}).

For this, we partition $\RR/\ZZ$ into arcs of length $\eps> 0$ so that
there are $\sim 1/\eps$ of them:
\begin{equation*}
  \RR/\ZZ = I_1 \cup I_2 \cup \dots \cup I_{\floor{1/\eps}},
\end{equation*}
where the $I_j$ are disjoint intervals of length at least $\eps$. 
\begin{lemma}\th\label{lem:avoidance}
  If $pz\in I_i$ and $p'z\in I_j$ for some $i\neq j$ and
  $I_i,I_j$ non-adjacent, then
  \begin{equation*}
    A_{P/p}^z\parens*{\frac{\eps}{2p}} \cap   A_{P/p'}^z\parens*{\frac{\eps}{2p'}}  = \emptyset,
  \end{equation*}
  where $P$ is a common multiple of $p, p'$.
\end{lemma}

\begin{proof}
  The non-adjacency of $I_i$ and $I_j$ implies that
  \begin{equation}
    \label{eq:nonadjacency}
    \norm{(p - p')z}\geq \eps.
  \end{equation}
  Meanwhile, the sets $A_{P/p}^z$ and $A_{P/p}^z$ intersect if and
  only if there exist integers $a,a'$ such that
  \begin{equation*}
    \abs*{\frac{a+z}{P/p} - \frac{a'+z}{P/p'}} < \frac{\eps}{P},
  \end{equation*}
  or, equivalently, if 
  \begin{equation*}
    \abs*{p'(a+z) - p(a' + z)} < \eps.
  \end{equation*}
  This last inequality implies that $\norm{(p - p')z} < \eps$,
  which is at odds with~(\ref{eq:nonadjacency}).
\end{proof}

\begin{proof}[Proof of~\th~\ref{modestmeasureofunions}]
  Let $0 < \eps < 1$ be a power of $2$, so that
  $\floor{1/\eps}=1/\eps$ is an even number. Note that no generality
  is lost in this restriction. For $X>1$,
  let
  \begin{equation*}
    P = \prod_{\substack{p \leq X \\ p\in\PP\cap N_\by(\by, \eps/2)}}p,
  \end{equation*}
  and write $q=P/p$. Note that for every $w\in\RR$, 
  \begin{equation*}
    \mu\parens*{A_q^w\parens*{\frac{\eps}{2p}}} = \frac{1}{P}\mu\parens*{R(\eps, w\bp, \bp)\cap [0,P]},
  \end{equation*}
  where $\bp$ is the vector whose entries are the prime divisors of
  $P$. Therefore, by~\th~\ref{thm:rogersext}, we have
  \begin{equation*}
    \mu\parens*{\bigcup_{\substack{p \leq X \\ p\in\PP\cap N_\by(\by,\eps/2) \\ p z \in I_j}} A_{q}^z\parens*{\frac{\eps}{2p}}}
    \geq \mu\parens*{\bigcup_{\substack{p \leq X \\ p\in \PP\cap N_\by(\by,\eps/2) \\ p z \in I_j}} A_{q}^0\parens*{\frac{\eps}{2p}}}.
  \end{equation*}
  By inclusion/exclusion, 
  \begin{align*}
    \mu\parens*{\bigcup_{\substack{p \leq X \\ p\in \PP\cap N_\by(\by,\eps/2) \\ p z \in I_j}} A_{q}^0\parens*{\frac{\eps}{2p}}}
    &= \eps \parens*{1 -  \prod_{\substack{p \leq X \\ p\in \PP\cap N_\by(\by,\eps/2) \\ pz \in I_j}}\parens*{1 - \frac{1}{p}}}.
  \end{align*}
  Note that
  \begin{align*}
    \prod_{\substack{p \leq X \\ p\in \PP\cap N_\by(\by,\eps/2) \\ pz \in I_j}}\parens*{1 - \frac{1}{p}}
    &=  \exp\sum_{\substack{p \leq X \\ p\in \PP\cap N_\by(\by,\eps/2) \\pz \in I_j}}\log\parens*{1 - \frac{1}{p}} \\
    &=  \exp\sum_{\substack{p \leq X \\ p\in \PP\cap N_\by(\by,\eps/2) \\ pz \in I_j}}\parens*{- \frac{1}{p} + O(p^{-2})} \\
    &= \exp\brackets*{-\sum_{\substack{p \leq X \\ p\in\PP\cap N_\by(\by,\eps/2) \\pz \in I_j}}\frac{1}{p}}e^{O(1)} \\
    &\leq \frac{1}{3} \qquad(X\gg 1),
  \end{align*}
  by~\th~\ref{lem:1/p}. In other words, by choosing $X>1$ sufficiently
  large, we ensure
  \begin{align*}
    \mu\parens*{\bigcup_{\substack{p \leq X \\ p \in \PP\cap N_\by(\by, \eps/2) \\ p z \in I_j}} A_{q}^z\parens*{\frac{\eps}{2p}}}
    &\geq \frac{2}{3} \eps
  \end{align*}
  for every $j=1, \dots, \floor{1/\eps} = 1/\eps$. Now, by~\th~\ref{lem:avoidance}, we have
  \begin{align*}
    \mu\parens*{\bigcup_{\substack{p \leq X \\ p\in \PP\cap N_\by(\by,\eps/2)}} A_{q}^z\parens*{\frac{\eps}{2p}}}
    &\geq\sum_{\substack{j=2 \\ j \textrm{ even}}}^{1/\eps} \mu\parens*{\bigcup_{\substack{p \leq X \\ p\in \PP\cap N_\by(\by,\eps/2) \\ p z \in I_j}} A_{q}^z\parens*{\frac{\eps}{2p}}} \\
    &\geq \parens*{\frac{1}{2\eps}}\parens*{\frac{2\eps}{3}} = \frac{1}{3}
  \end{align*}
  as required. This settles~(\ref{eq:6}). 

  Meanwhile, by~\th~\ref{lem:smallmeasure},  
  \begin{equation*}
    \mu\parens*{\bigcup_{\substack{p \leq X \\ p\in \PP\cap N_\by(\by,\eps/2)}} A_{q}^y\parens*{\frac{\eps}{2p}}}
    \leq \mu\parens*{A_P^y(\eps)} = 2\eps,
  \end{equation*}
  for every $y\in \set{y_1, \dots, y_k}$, which
  establishes~(\ref{eq:5}) and finishes the proof.
\end{proof}

\appendix

\section{On subsets whose reciprocal sums diverge \\ by Manuel
  Hauke}\label{sec:manuel-hauke}

\th~\ref{lem:1/p} motivates the following question, stated
as~\th~\ref{ntproblem} in Section~\ref{sec:furth-disc-quest}:
\emph{Does every positive-density set $A$ contain a subset
  $B\subset A$ such that}
\begin{equation}\label{prop_B}
  \sup_{\substack{n,m \in B\\n \neq m}}\gcd(m,n) < \infty \quad\textrm{and}\quad \sum_{n \in B}\frac{1}{n} = \infty?
\end{equation}
We answer the question in the negative.
  
In fact, we examine \eqref{prop_B} for arbitrary sets $A$.  Clearly,
if $\mathbb{P} \subseteq A$, then \eqref{prop_B} holds for
$B = \mathbb{P}$. Similarly, if $n_0\mathcal{P} \subseteq A$ for some
$n_0 \in \mathbb{N}$ and $\mathcal{P} \subseteq \mathbb{P}$ such that
$\sum_{p \in \mathcal{P}} \frac{1}{p} = \infty$, then~(\ref{prop_B})
holds for $B=n_0 \calP$. We show that this construction is not only
sufficient but also necessary for the existence of $B\subset A$
satisfying \eqref{prop_B}:

\begin{theorem}\th\label{thm_charact}
  An arbitrary set of integers $A$ contains a subset $B$ satisfying
  \eqref{prop_B} if and only if there exists $n_0 \in \mathbb{N}$ and 
  $\mathcal{P} \subseteq \mathbb{P}$ such that $n_0\mathcal{P} \subseteq A$
  and $\sum_{p \in \mathcal{P}} \frac{1}{p} = \infty$.
\end{theorem}

Using \th\ref{thm_charact}, we answer~\th~\ref{ntproblem}
negatively in the strongest possible form, by showing that the
complement of $A$ can be essentially as sparse as the primes, while
$A$ does not contain a set $B$ satisfying~(\ref{prop_B}).

\begin{corollary}\th\label{neg_answer}
  For any monotonically increasing and doubling function
  $f: \mathbb{N} \to [0,\infty)$ with
  $\lim_{x \to \infty} f(x) = \infty$, there exists $A$ which does not
  have a subset $B$ satisfying \eqref{prop_B}, while
  \begin{equation}\label{eq:lowdense}
    \lim_{x \to \infty}\frac{\#\{n \leq x: n \in \NN\setminus A \}}{x f(x)/\log(x)} = 0.
  \end{equation}
  In particular, there exists such a set $A$ with asymptotic density
  $1$.
\end{corollary}

\begin{proof}[Proof of Theorem \ref{thm_charact}]
  If there exists $n_0 \in \mathbb{N}$,
  $\mathcal{P} \subseteq \mathbb{P}$ with $n_0\mathcal{P} \subseteq A$
  and $\sum_{p \in \mathcal{P}} \frac{1}{p} = \infty$, then we pick
  $B = n_0\mathcal{P}$, and see that \eqref{prop_B} is satisfied. Thus
  it remains to show the converse. For this, let $A$ be fixed and let
  \eqref{prop_B} hold for some set $B \subseteq A$.  Denote
  \begin{equation}\label{def_M}
    M := \sup_{\substack{n,m \in B\\n \neq m}}\gcd(m,n).
  \end{equation}
  Then every $n\in B$ may be written as
  \begin{equation*}
    n = \underbrace{\parens*{\prod_{p \leq M} p^{\nu_p(n)}}}_{m_n}  \parens*{\prod_{p > M} p^{\nu_p(n)}},
  \end{equation*}
  and~(\ref{def_M}) implies that for every prime $p\leq M$, at most
  one member of $\set{m_n : n\in B}$ is divisible by $p^M$. Therefore,
  by the pigeonhole principle, the set $\set{m_n : n\in B}$ is finite,
  and it follows that it contains and element $n_0$ such that
  \begin{equation}\label{eq:Bprimediv}
    \sum_{n\in B'} \frac{1}{n} = \infty,
  \end{equation}
  where
  \begin{equation*}
    B' = \set*{n\in B : n=n_0  \prod_{p > M} p^{\nu_p(n)} \textrm{ and } n\neq n_0}.
  \end{equation*}
  Notice that for each prime $p>M$, there is at most one $n\in B'$
  such that $p\mid n$, by~(\ref{def_M}). Therefore,
  \begin{equation*}
    p_{\min}(n):=\min\set{p> M: p\mid n}
  \end{equation*}
  defines an injection $B'\hookrightarrow \PP\cap(M,\infty)$. Now
  \begin{equation*}
     \sum_{\substack{n\in B' \\ n/n_0 \textrm{ is composite}}} \frac{1}{n} \leq \frac{1}{n_0}\sum_{\substack{n\in B' \\ n/n_0 \textrm{ is composite}}} \frac{1}{p_{\min}(n)^2} \leq \frac{1}{n_0}\sum_{p\in\PP}\frac{1}{p^2} < \infty,
   \end{equation*}
   and it follows from~(\ref{eq:Bprimediv}) that
  \begin{equation*}
    \sum_{\substack{n\in B' \\ n/n_0 \textrm{ is prime}}} \frac{1}{n} = \infty. 
  \end{equation*}
  The theorem is proved by putting $\calP = \PP\cap \set{n/n_0 : n\in B'}$. 
\end{proof}

\begin{proof}[Proof of Corollary \ref{neg_answer}]
  Given $f$ as in the statement of the corollary, we let
  $f_m := f(2^m)$ and define
  \[S_m:= \{n \in [2^m,2^{m+1}): n = p\cdot c, p \in \mathbb{P}, 1
    \leq c \leq \sqrt{f_m}\}.\] By a weak form of the Prime Number
  Theorem, we get for $m$ sufficiently large,
  \[\#S_m \leq \sum_{1 \leq c \leq \sqrt{f_m}} \#\left(\mathbb{P} \cap
      [2^m/c,2^{m+1}/c)\right) \ll \frac{\sqrt{f_m}2^m}{m}.
    \]
    Setting now $S := \bigcup_{m \geq 1}S_m$, we get that for
    $2^m \leq N < 2^{m+1}$ that
    \[
      \#\{n \leq N: n \in S\} \ll \sum_{j \leq
        m}\frac{\sqrt{f_j}2^j}{j} \ll \frac{\sqrt{f_m}2^m}{m} \ll
      \frac{\sqrt{f(N)}N}{\log N},
    \]
    where the doubling assumption on $f$ is used in the last step. In
    particular,~(\ref{eq:lowdense}) holds by putting
    $A = \NN \setminus S$. It remains to show that $A$ does not
    contain a set $B$ satisfying \eqref{prop_B}. Assuming the
    contrary, an application of Theorem \ref{thm_charact} provides the
    existence of $n_0\in\NN$ and an infinite set $\calP\subset\PP$
    such that $n_0 \mathcal{P} \subseteq A$. However, since
    $f_m\to \infty\, (m\to\infty)$, we have for sufficiently large
    $p\in\PP$ that $n_0p \in S_m$ for some $m \in \mathbb{N}$. Thus
    $\mathcal{P}$ must be finite, contradicting Theorem
    \ref{thm_charact}.

    Finally, the ``in particular'' part of the corollary follows by
    choosing $f(x) := \log x$.
\end{proof}

\subsection*{Acknowledgments}

SMH and FAR thank Sam Chow for pointing out a historical inaccuracy in
the first version of this paper, and Manuel Hauke for generously
contributing the appendix.

\bibliographystyle{plain}

\end{document}